\documentclass{amsart}

\usepackage{color}
\usepackage{tikz-cd}

\newtheorem{Theorem}{\bf Theorem}
\newtheorem{lemma}[Theorem]{\bf Lemma}
\newtheorem{proposition}[Theorem]{\bf Proposition}

\newtheorem{remark}[Theorem]{\bf Remark}
\newtheorem{theorem}[Theorem]{\bf Theorem}

\usepackage[all,cmtip]{xy}

     \title[On characterising analytically unramified local rings]{On characterising analytically unramified local rings}
      \author{Vijay Kodiyalam}
     \address{The Institute of Mathematical Sciences, Chennai, India and Homi Bhabha National Institute, Mumbai, India}
     \email{vijay@imsc.res.in}
     \author{J. K. Verma}
     \address{Indian Institute of Technology Gandhinagar, Palaj, Gandhinagar 382355, Gujarat, India}
     \email{jugal.verma@iitgn.ac.in}


\begin{document}
     \begin{abstract}
     We generalise a classic result of Rees to characterise analytically unramified local rings using Rees algebras of modules.
      \end{abstract}
     \maketitle

A famous result of Rees in \cite{Res1961} characterises  analytically unramified Noetherian local rings $(R,{\mathfrak m})$ as those for which, for some ${\mathfrak m}$-primary ideal $I$ (equivalently, for all ideals $I$),  the integral closure of the Rees algebra $R[It]$ in $R[t]$ is a finitely-generated $R[It]$-module.
In this note, we generalise this result to obtain an analogous characterisation related to finiteness of integral closure of Rees algebras of modules. Our result is the following theorem.

\begin{theorem} \label{theorem-main} Let $(R,{\mathfrak m},k)$ be a Noetherian local ring. The following conditions are equivalent.
 \begin{enumerate}
 \item $R$ is analytically unramified.
  \item For any submodule $M \subseteq F$ with $F$ finitely-generated free, the integral closure of $S(M)$ in $S(F)$ is a finitely-generated $S(M)$-module.
  \item There exists a non-zero finite length module $Q$ and a finitely-generated free module $F$ mapping onto $Q$ with kernel $M$ such that the integral closure of $S(M)$ in $S(F)$ is a finitely-generated $S(M)$-module.
 \end{enumerate}
 \end{theorem}    
 
 The notations $S(F)$ and $S(M)$ in the theorem refer to the symmetric algebra of $F$ (which is a polynomial ring over $R$) and to the image of the symmetric algebra of $M$ in $S(F)$, both of which are ${\mathbb N}$-graded $R$-algebras.
 The algebra $S(M)$ is the Rees algebra of $M \subseteq F$ and generalises the Rees algebra of an ideal.
 
 We will use two facts about graded rings. The first is that the integral closure of an ${\mathbb N}$-graded ring in another is also compatibly ${\mathbb N}$-graded - see Theorem 2.3.2 of \cite{SwnHnk2006} for a proof. The second is the following elementary lemma whose proof we omit.
 
 \begin{lemma}\label{lemma-fingen} Let $A \subseteq B \subseteq C$ be ${\mathbb N}$-graded rings. Suppose that (i) $C_0$ is a finitely-generated $A_0$-module, (ii) $A=A_0[A_1]$, and (iii) $A$ and $C$ are Noetherian. Then the following conditions are equivalent:
\begin{enumerate}
\item $B$ is finitely-generated as an $A$-module.
\item There is a $k \geq 0$ such that for all $n \geq k$, $B_{n} \subseteq A_{n-k}C_k$.
\end{enumerate} 
\end{lemma}

Products such as $A_{n-k}C_k$ refer to the $A_0$-submodule of $C_{n}$ generated by products of all pairs of elements of $A_{n-k}$ and $C_k$. Thus, for instance, $A_n = (A_1)^n$ since $A = A_0[A_1]$.

The next proposition is a special case of the implication $(1) \Rightarrow (2)$ of Theorem~\ref{theorem-main}. Here and in the sequel, we will adopt the following notation. Let $V$ denote the integral closure of $S(M)$ in $S(F)$, which is a compatibly graded intermediate subalgebra. Thus, $V = \bigoplus_{n \geq 0} V_n$ with $S_n(M) \subseteq V_n \subseteq S_n(F)$.

We will also use the notion and properties of a Nagata ring - a Noetherian ring $R$ for every prime $P$ of which, the integral closure of $\frac{R}{P}$ in a finite extension of its field of fractions is a finitely-generated $R$-module. A  Noetherian complete local ring is a Nagata ring - see Chapter 12 of \cite{Mts1980}. We will also need to use the fact that if $R$ is a Nagata ring and $R \rightarrow S$ is essentially of finite type, i.e., $S$ is a localisation of a finitely-generated $R$-algebra, then the integral closure of $R$ in $S$ is a finitely-generated $R$-module. To prove this, very briefly, one embeds $S$ in the product of all $S_Q$ where $Q$ is a minimal prime ideal of $S$, reduces to proving the statement for a single such $S_Q$ - which is a field that is finitely-generated over the field of fractions of $\frac{R}{P}$ where $P$ is the contraction of $Q$ to $R$ - and then uses that the algebraic closure of the field of fractions of $\frac{R}{P}$ in $S_Q$ is actually a finite extension.

\begin{proposition} \label{prop-crnlr}Let $(R,{\mathfrak m},k)$ be a complete, local, reduced Noetherian ring. Then for any submodule $M \subseteq F$ with $F$ finitely-generated free, the integral closure of $S(M)$ in $S(F)$ is a finitely-generated $S(M)$-module.
\end{proposition}

\begin{proof} We have $S(M) = \bigoplus_{n \geq 0} S_n(M) \subseteq \bigoplus_{n \geq 0} V_n \subseteq \bigoplus_{n \geq 0} S_n(F) = S(F)$.
Define  the algebras
$$
\widehat{S(M)} = \prod_{n \geq 0} S_n(M) \subseteq \prod_{n \geq 0} S_n(F) = \widehat{S(F)}.
$$
Choose a basis $\{T_1,\cdots,T_r\}$ of $F$ and a generating set $\{L_1,\cdots,L_d\}$ of $M$. We will regard the $L_j$ as linear forms in $\{T_1,\cdots,T_r\}$ with coefficients in $R$. We then have natural identifications of algebras
$$
\begin{array}{ccccc}
S(M) = R[L_1,\cdots,L_d] & \subseteq &  R[T_1,\cdots,T_r] = S(F)\\
\ \ \ \ \ \ \ \ \ \ \ \ \rotatebox[origin=c]{270}{$\subseteq$} &  & \!\!\!\!\!\!\!\!\!\!\!\!\!\!\!\!\!\!\rotatebox[origin=c]{270}{$\subseteq$}\\
\widehat{S(M)} = {R}[[L_1,\cdots,L_d]] &  \subseteq & {R}[[T_1,\cdots,T_r]] = \widehat{S(F)}.
\end{array}
$$

The algebra $\widehat{S(M)}$ is the image of $R[[Z_1,\cdots,Z_d]]$ under the (continuous) homomorphism to ${R}[[T_1,\cdots,T_r]]$ taking $T_j$ to $L_j$ and is therefore a Noetherian complete local ring. Hence $\widehat{S(M)}$ is a  Nagata ring, 
and since $\widehat{S(M)}[T_1,\cdots,T_r] \subseteq  \widehat{S(F)}$ is reduced and is a finitely-generated $\widehat{S(M)}$-algebra,  the integral closure of $\widehat{S(M)}$ in $\widehat{S(M)}[T_1,\cdots,T_r] $ is a finitely-generated $\widehat{S(M)}$-module.

The intermediate subalgebra   $\widehat{S(M)}[T_1,\cdots,T_r]$ of $\widehat{S(M)} \subseteq  \widehat{S(F)}$ has an ascending filtration by  $\widehat{S(M)}$-submodules generated by all polynomials of degree at most $k$ in $T_1,\cdots,T_r$ with $\widehat{S(M)}$ coefficients. Denote this submodule by $F_k\left(\widehat{S(M)}[T_1,\cdots,T_r]\right)$. 
A little thought shows that this submodule has an alternative description as
$$
\left(\prod_{0 \leq n < k} S_n(F)\right) \times \left( \prod_{n \geq k} S_{n-k}(M)S_k(F)\right) \subseteq \prod_{n \geq 0} S_n(F) =  \widehat{S(F)}.
$$
Thus, $\widehat{S(M)}[T_1,\cdots,T_r]$ is
$$
\bigcup_{k \geq 0} \left( \left(\prod_{0 \leq n < k} S_n(F)\right) \times \left( \prod_{n \geq k} S_{n-k}(M)S_k(F)\right)\right) \subseteq \prod_{n \geq 0} S_n(F),
$$
an ascending union of finitely-generated $\widehat{S(M)}$-submodules.

Since the integral closure of $\widehat{S(M)}$ in $\widehat{S(M)}[T_1,\cdots,T_r]$ is a finitely-generated $\widehat{S(M)}$-module, it follows that it is contained in $F_k\left(\widehat{S(M)}[T_1,\cdots,T_r]\right)$
 for some $k \geq 0$. Hence the integral closure $V$ of $S(M)$ in $S(M)[T_1,\cdots,T_r] = S(F)$ is also contained in $F_k\left(\widehat{S(M)}[T_1,\cdots,T_r]\right)$. In particular,
$$
V_n \subseteq F_k\left(\widehat{S(M)}[T_1,\cdots,T_r]\right) \cap S_n(F) = S_{n-k}(M)S_k(F),
$$
for all $n \geq k$, where the last equality follows from the alternative description of $F_k\left(\widehat{S(M)}[T_1,\cdots,T_r]\right)$.
Finally, an application of Lemma \ref{lemma-fingen} to $S(M) \subseteq V \subseteq S(F)$ shows that $V$ is a finitely-generated $S(M)$-module, as desired.
\end{proof}

The next few results will be needed in the proof of the implication $(3) \Rightarrow (1)$ of Theorem~\ref{theorem-main}.
The next lemma is Lemma~1 of \cite{Res1961} which gives the easier direction of his characterisation result. 
The notation $\bar{J}$ for an ideal $J$ denotes, as usual, the integral closure of $J$.

\begin{lemma}\label{rees}
Suppose that $(R,{\mathfrak m},k)$ is a Noetherian local ring and $I$ is an ${\mathfrak m}$-primary ideal of $R$ such that
$\overline{I^n} \subseteq I^{m(n)}$ where $m(n)$ goes to infinity as $n$ goes to infinity. Then $R$ is analytically unramified.
\end{lemma}

Before stating the next lemma, we will explain the notation used. Let $R$ be a Noetherian ring and $M \subseteq F$ be $R$-modules with $F$ finitely-generated and free. Let $I(M)$ denote the ideal of maximal minors of a matrix whose columns are generators of $M$ expressed in terms of a basis of $F$. More generally, let $I(S_n(M))$ be the ideal of maximal minors of a matrix whose columns are generators of $S_n(M)$ expressed in terms of a basis of $S_n(F)$. Recall that $V$ is the integral closure of $S(M)$ in $S(F)$.

\begin{lemma} \label{ann}
Suppose that $M \subseteq F$ are modules over a Noetherian ring $R$ with $F$ finitely-generated and free. 
Then for each $n \geq 1$, $\overline{I(S_n(M))} S_n(F) \subseteq V_n$. 
\end{lemma}

\begin{proof} Choose a basis $\{T_1,\cdots,T_r\}$ of $F$ and a generating set $\{L_1,\cdots,L_d\}$ of $M$ and consider the $r \times d$ matrix, say $A$,  whose columns are the coefficients of $T_1,\cdots,T_r$ in $L_1,\cdots,L_d$. A typical generator of $I(M)$ is the determinant of an $r  \times r$ submatrix, say $C$, of $A$. The equation $det(C)I = C.adj(C)$ shows that $det(C)F \subseteq  M$. Explicitly, $det(C)T_k = d_{1k}C_1 + d_{2k}C_2 + \cdots + d_{rk}C_r$ where $C_1,\cdots,C_r$ are the columns of $C$ (which are some of the $L_1,\cdots,L_d$) and $d_{1k},\cdots,d_{rk}$ are the entries of the $k$-th column of $adj(C)$.
Hence $I(M)F \subseteq M$. More generally, applying similar reasoning to $S_n(M) \subseteq S_n(F)$, we get $I(S_n(M))S_n(F) \subseteq S_n(M)$.

Now we need to see that for any element $x \in \overline{I(S_n(M))}$ and any monomial, say $Z$, of degree $n$ in $T_1,\cdots,T_r$,  the element $xZ \in S_n(F)$ satisfies an integral equation over $S(M)$ (and is consequently is in $V_n$). Suppose that 
$$
x^p+a_1x^{p-1}+\cdots+a_p = 0
$$
is an integral equation for $x$ over $I(S_n(M))$ where $a_j \in I(S_n(M))^j$ for $1 \leq j \leq p$. Multiply this equation by $Z^p$ to get
$$
(xZ)^p+(a_1Z)(xZ)^{p-1}+\cdots+(a_pZ^p) = 0.
$$
We assert that each $a_jZ^j$ is in $S_{jn}(M)$ thereby showing that $xZ$ is integral over $S(M)$.

To see this, note that each $a_j$ is a linear combination, with coefficients in $R$, of products of the form $b_1b_2\cdots b_j$ where the $b_k \in I(S_n(M))$. Since each $b_kZ \in S_n(M)$ (as $I(S_n(M))S_n(F) \subseteq S_n(M)$), it follows that $a_jZ^j \in S_{jn}(M)$, as needed.
\end{proof}

\begin{lemma} \label{lemma-firstentry} Suppose that $M \subseteq F$ over a Noetherian local ring $(R,{\mathfrak m},k)$ with $F$ free and finitely-generated and $Q = \frac{F}{M}$ of finite length and non-zero. Then, there is a basis  $\{T_1,\cdots,T_r\}$ of $F$  and a set of generators $\{L_1,\cdots,L_d\}$ of $M$ such that the coefficient of $T_1$ in each $L_j$ is in ${\mathfrak m}$.
\end{lemma}

\begin{proof} Filtering $Q$ by copies of $k$, the residue field of $R$, we may enlarge $M$ so that $Q \cong k$. The isomorphism of $Q$ to $k$ gives a map of $F$ onto $k = \frac{R}{\mathfrak m}$ which can be lifted to a map from $F$ to $R$ which is necessarily onto. This map, say $f_1: F \rightarrow R$, is a basis element of $F^*$ and can be completed to a basis $\{f_1,\cdots,f_r\}$ of $F^*$. Let $\{T_1,\cdots,T_r\}$ be the dual basis of $F$, so that $f_i(\cdot)$ gives the coefficient of $T_i$. Since $f_1$ is a lift of an isomorphism of $ \frac{F}{M}$ to $k$, necessarily $f_1(M) \subseteq {\mathfrak m}$, or equivalently, the coefficient of $T_1$ in each $L_j$ is in ${\mathfrak m}$.
\end{proof}

\begin{proof}[Proof of Theorem \ref{theorem-main}] 
We will prove the two non-trivial implications.\\
$(1) \Rightarrow (2)$: Choose a basis $\{T_1,\cdots,T_r\}$ of $F$ and a generating set $\{L_1,\cdots,L_d\}$ of $M$. 
Let $V$ be the integral closure of $S(M)$ in $S(F)$, which we need to show is a finitely-generated $S(M)$-module.

With $\widehat{R}$ denoting the ${\mathfrak m}$-adic completion of $R$ (which is reduced since $R$ is analytically unramified), let $\widehat{M} = \widehat{R} \otimes_R M \subseteq \widehat{R} \otimes_R F = \widehat{F}$. Then, $S(\widehat{F})$ is identified with $\widehat{R}[T_1,\cdots,T_r]$
and $S(\widehat{M})$ with its subalgebra $\widehat{R}[L_1,\cdots,L_d]$. Let $W$ be the integral closure of $S(\widehat{M})$ in $S(\widehat{F})$ which is a graded intermediate algebra. There are then natural inclusions of graded algebras as shown below.

$$
\begin{array}{ccccc}
S(M) = R[L_1,\cdots,L_d] & \subseteq & V = \bigoplus_{n \geq 0} V_n & \subseteq & R[T_1,\cdots,T_r] = S(F)\\
\ \ \ \ \ \ \ \ \ \ \ \ \rotatebox[origin=c]{270}{$\subseteq$} & & \!\!\!\!\rotatebox[origin=c]{270}{$\subseteq$} & & \!\!\!\!\!\!\!\!\!\!\!\!\!\!\!\!\!\!\rotatebox[origin=c]{270}{$\subseteq$}\\
S(\widehat{M}) = \widehat{R}[L_1,\cdots,L_d] & \subseteq  & W = \bigoplus_{n \geq 0} W_n & \subseteq & \widehat{R}[T_1,\cdots,T_r] = S(\widehat{F})
\end{array}
$$

Suppose that $W$ is a finitely-generated $S(\widehat{M})$-module. By Lemma \ref{lemma-fingen}, there is a $k$ such that for all $n \geq k$, $W_{n} \subseteq (S(\widehat{M}) )_{n-k}(S(\widehat{F}) )_k$. From the picture above it follows that
$V_{n} \subseteq (S(\widehat{M}) )_{n-k}(S(\widehat{F}) )_k \cap S(F)_{n} = (S({M}) )_{n-k}(S({F}) )_k$, where the last equality follows from the faithful flatness of $\widehat{R}$ over $R$. By Lemma \ref{lemma-fingen} again, $V$ is a finitely-generated $S(M)$-module. This reduces proving $(2)$ to reduced complete local rings which follows from Proposition \ref{prop-crnlr}.\\
$(3) \Rightarrow (1)$: Given $M \subseteq F$ with the integral closure $V$ of $S(M)$ in $S(F)$ being a finitely-generated $S(M)$-module, we will show that the ideal $I = I(M)$ satisfies the hypothesis of Lemma \ref{rees} and thereby conclude that $R$ is analytically unramified. First, since $Q = \frac{F}{M}$ is of finite length, it vanishes on localisation at any non-maximal prime $P$ of $R$. Hence $IR_P = R_P$, and  so $I$ is not contained in $P$. Hence $I$ is ${\mathfrak m}$-primary.

Since $V$ is a finitely-generated $S(M)$-module, by Lemma \ref{lemma-fingen} there is a $k \geq 0$ such that for all $n \geq k$, $V_n  \subseteq S_{n-k}(M)S_k(F) \subseteq S_n(F)$. By Lemma \ref{ann}, $\overline{I(S_n(M))}S_n(F) \subseteq V_n$.

Now, by Theorem 1 of \cite{BrnVsc2003}, $I(S_n(M)) = I^{\binom{n+r-1}{r}}$  up to integral closure. Hence,
$$
\overline{I^{\binom{n+r-1}{r}}}S_n(F) \subseteq V_n \subseteq S_{n-k}(M)S_k(F).
$$
This says that $\overline{I^{\binom{n+r-1}{r}}}$ is contained in the annihilator of $\frac{S_n(F)}{S_{n-k}(M)S_k(F)}$.
Next, we will show that for any $m$, this annihilator is contained in $I^m$ for all $n$ sufficiently large. 

Choose a basis $\{T_1,\cdots,T_r\}$ of $F$ and a set of generators $\{L_1,\cdots,L_d\}$ of $M$ as in Lemma \ref{lemma-firstentry}. Then, 
$S_n(F)$ is identified with the free module of homogeneous forms of degree $n$ in $T_1,\cdots,T_r$ with coefficients in $R$ and its submodule $S_{n-k}(M)S_k(F)$ is identified with the submodule generated by those forms that are products of $n-k$ of the $L_1,\cdots,L_d$ with any monomial of degree $k$ in $T_1,\cdots,T_r$. The coefficient of $T_1^n$ in any generator of 
$S_{n-k}(M)S_k(F)$ is therefore in ${\mathfrak m}^{n-k}$. It follows that the annihilator of $\frac{S_n(F)}{S_{n-k}(M)S_k(F)}$ is contained in ${\mathfrak m}^{n-k}$. Since $I$ is ${\mathfrak m}$-primary, 
given any $m$, there exists an $n(m)$ such that 
$$
\overline{I^{\binom{n+r-1}{r}}} \subseteq I^m.
$$
for all $n \geq n(m)$. In fact, if ${\mathfrak m}^t \subseteq I$, we may take $n(m) = mt+k$.

Define $m(n)=0$ for $n < \binom{n(1)+r-1}{r}$, $m(n) = 1$ for $\binom{n(1)+r-1}{r} \leq n < \binom{n(2)+r-1}{r}$ and so on. Then, for every $n$, $$
\overline{I^n} \subseteq I^{m(n)}.
$$
Clearly $m(n)$ goes to infinity as $n$ does, and so by Lemma \ref{rees}, $R$ is analytically unramified.
\end{proof}
     
\begin{remark}[Referee's remark]\label{refereeremark}
In Example 6 of the appendix of \cite{Ngt1962}, a  regular local ring $R$ and an algebra $S=R[d]$ are constructed so  that  $S$  is a normal analytically ramified local ring. By Rees's Theorem, there exists an ideal $I$ of $S$ so that the normalization of the Rees algebra $S[It]$  in $S[t]$ (where $t$ is an indeterminate) is not a finitely generated $S[It]$-module. This shows the non-triviality of the implication $(1) \Rightarrow (2)$ in Theorem 1, by showing that $S(M)$ cannot be replaced by an arbitrary finitely-generated $R$-algebra in that theorem.
\end{remark}  

\section*{Acknowledgements}
\noindent
The authors thank the referee for a careful reading and especially for Remark \ref{refereeremark}.

\end{document}